\documentclass[12pt]{article}
\usepackage[a4paper,margin=1.1in]{geometry}
\usepackage{amssymb}
\usepackage{amsmath}
\usepackage{amsfonts}
\usepackage{amsthm}
\usepackage{color}
\usepackage{float}
\usepackage[all]{xy}
\usepackage{scalerel}
\usepackage{mathabx}
\usepackage{delimset}
\usepackage{setspace}

\newcommand{\qbinom}{\genfrac{[}{]}{0pt}{}}

\def\RR{\mathop{\mbox{\textup{R}}}\nolimits}

\newcommand{\C}{\mathbb{C}}

\newcommand{\N}{\mathbb{N}}

\newcommand{\R}{\mathbb{R}}
\newcommand{\T}{\mathrm{T}}

\newtheorem{theorem}{Theorem}

\begin{document}

\title{\textbf{New $q$-identities Via $q$-Derivative of Basic Hypergeometric Series with Respect to Parameters}}
\author{Ronald Orozco L\'opez}
\newcommand{\Addresses}{{
  \bigskip
  \footnotesize

  \textit{E-mail address}, R.~Orozco: \texttt{rj.orozco@uniandes.edu.co}
  
}}

\maketitle

\begin{abstract}
In this paper, we use the effect of the $q$-differential and deformed $q$-exponential operators on basic hypergeometric series to find new $q$-identities from the $q$-Gauss sum, the $q$-Chu-Vandermonde's sum, and Jackson's transformation formula.
\end{abstract}


\section{Introduction and Motivation}

Throughout this paper, we take $0<\vert q\vert<1$. The $q$-differential operator $D_{q}$ is defined by:
\begin{equation*}
    D_{q}f(x)=\frac{f(x)-f(qx)}{x}
\end{equation*}
and the Leibniz rule for $D_{q}$
\begin{equation}\label{eqn_leibniz}
    D_{q}^{n}\{f(x)g(x)\}=\sum_{k=0}^{n}q^{k(k-n)}\qbinom{n}{k}_{q}D_{q}^{k}\{f(x)\}D_{q}^{n-k}\{g(q^{k}x)\},
\end{equation}
where the $q$-binomial coefficient is defined by
\begin{equation*}
\qbinom{n}{k}_{q}=\frac{(q;q)_{n}}{(q;q)_{k}(q;q)_{n-k}},
\end{equation*}
and the $q$-shifted factorial of $a$ is given by $(a;q)_{0}=1$ and
\begin{align*}
    (a;q)_{n}&=\prod_{k=0}^{n-1}(1-q^{k}a),\text{ for }n=1,2,\ldots,\\
    (a;q)_{\infty}&=\lim_{n\rightarrow\infty}(a;q)_{n}=\prod_{k=0}^{\infty}(1-aq^{k}).
\end{align*}
Following Gasper and Rahman \cite{gasper}, the ${}_r\phi_{s}$ basic hypergeometric series is defined by
\begin{equation*}
    {}_r\phi_{s}\left(
    \begin{array}{c}
         a_{1},a_{2},\ldots,a_{r} \\
         b_{1},\ldots,b_{s}
    \end{array}
    ;q,z
    \right)=\sum_{n=0}^{\infty}\frac{(a_{1},a_{2},\ldots,a_{r};q)_{n}}{(q;q)_{n}(b_{1},b_{2},\ldots,b_{s};q)_{n}}\Big[(-1)^{n}q^{\binom{n}{2}}\Big]^{1+s-r}z^n
\end{equation*}
where
\begin{align*}
    (a_{1},a_{2},\ldots,a_{m};q)_{n}&=(a_{1};q)_{n}(a_{2};q)_{n}\cdots(a_{m};q)_{n}.
\end{align*}
Equivalently, 
\begin{equation*}
    (a_{1},a_{2},\ldots,a_{m};q)_{\infty}=(a_{1};q)_{\infty}(a_{2};q)_{\infty}\cdots(a_{m};q)_{\infty}.
\end{equation*}
In \cite{orozco}, we define the deformed basic hypergeometric series ${}_{r}\Phi_{s}$ as
    \begin{equation}\label{eqn_def_hyp}
        {}_{r}\Phi_{s}\left(
    \begin{array}{c}
         a_{1},a_{2},\ldots,a_{r} \\
         b_{1},\ldots,b_{s}
    \end{array}
    ;q,u,z
    \right)=\sum_{n=0}^{\infty}u^{\binom{n}{2}}\frac{(a_{1},a_{2},\ldots,a_{r};q)_{n}}{(q,b_{1},b_{2},\ldots,b_{s};q)_{n}}[(-1)^nq^{\binom{n}{2}}]^{1+s-r}z^n.                
    \end{equation}
If $u=1$, then ${}_{r}\Phi_{s}={}_{r}\phi_{s}$. If $u=q$, 
\begin{equation*}
    {}_{r+1}\Phi_{r}\left(
    \begin{array}{c}
         a_{1},a_{2},\ldots,a_{r},0 \\
         b_{1},\ldots,b_{r}
    \end{array}
    ;q,q,z
    \right)={}_{r}\phi_{r}\left(
    \begin{array}{c}
         a_{1},a_{2},\ldots,a_{r} \\
         b_{1},b_{2},\ldots,b_{r}
    \end{array}
    ;q,-z
    \right).
\end{equation*}

In this paper, we will frequently use the following identities:
\begin{align}
    (a;q)_{n}&=\frac{(a;q)_{\infty}}{(aq^n;q)_{\infty}},\label{eqn_iden1}\\
    (a;q)_{n+k}&=(a;q)_{n}(aq^{n};q)_{k},\label{eqn_iden2}\\
    (aq^{k};q)_{n-k}&=\frac{(a;q)_{k}}{(a;q)_{i}},\label{eqn_iden3}\\
    (aq^n;q)_{k}&=\frac{(a;q)_{k}(aq^k;q)_{n}}{(a;q)_{n}}.\label{eqn_iden6}
\end{align}
In addition, we will use the identities for binomial coefficients:
\begin{align*}
    \binom{n+k}{2}&=\binom{n}{2}+\binom{k}{2}+nk,\\
    \binom{n-k}{2}&=\binom{n}{2}+\binom{k}{2}+k(1-n).
\end{align*}
This paper provides new $q$-identities by applying the $q$-differential and deformed $q$-exponential operators to the $q$-Gauss sum, the $q$-Chu-Vandermonde's sum, and Jackson's transformation formula. These q-identities have the following quotient forms of the q-shifted factorials
\begin{equation*}
    \frac{1}{(c;q)_{n}},\ \ \frac{x^n}{(c;q)_{n}},\ \ \frac{(ac;q)_{n}}{(bc;q)_{n}},\text{ and }\frac{(ac,bc;q)_{n}}{(dc,ec;q)_{n}}.
\end{equation*}
The method will consist of applying the operators $D_{q}$ and $\T(yD_{q}|u)$, with respect to the parameter $c$, to both sides of the above $q$-identities and then using the identities Eqs.(\ref{eqn_iden1})-(\ref{eqn_iden6}). Something similar was done in \cite{ghany}.

\section{Some identities of the $q$-differential operator}

For all $n\in\N$, we have the identities:
\begin{align}\label{eqn_iden5}
    D_{q}^kx^{n}=\begin{cases}
    \frac{(q;q)_{n}}{(q;q)_{n-k}}x^{n-k},&\text{ if } k\leq n;\\
    0,&\text{ if }k>n.
    \end{cases}
\end{align}

\begin{align}\label{eqn_iden4}
    D_{q}^{k}(ax;q)_{n}=
    \begin{cases}
        (-a)^kq^{\binom{k}{2}}\frac{(q;q)_{n}}{(q;q)_{n-k}}\frac{(ax;q)_{n}}{(ax;q)_{k}},&\text{ if }k\leq n;\\
        0,&\text{ if }k>n.
    \end{cases}
\end{align}

\begin{theorem}\label{theo_nder_geo}
For all $n\geq1$,
    \begin{equation}\label{eqn_nder_geo}
        D_{q}^{n}\left(\frac{1}{1-ax}\right)=\frac{(q;q)_{n}a^n}{(ax;q)_{n+1}}.
    \end{equation}
\end{theorem}
\begin{proof}
Note that
\begin{equation}
    D_{q}\left(\frac{1}{1-ax}\right)=\frac{(1-q)a}{(1-ax)(1-aqx)}=\frac{(q;q)_{1}a}{(ax;q)_{2}}.
\end{equation}
Suppose that Eq.(\ref{eqn_nder_geo}) is true for $n$ and let us prove by induction for $n+1$. As
\begin{equation}
    D_{q}(ax;q)_{n+1}=-(1-q^{n+1})a(aqx;q)_{n},
\end{equation}
then
\begin{align*}
    D_{q}^{n+1}\left(\frac{1}{1-ax}\right)&=D_{q}D_{q}^{n}\left(\frac{1}{1-ax}\right)\\
    &=D_{q}\left(\frac{(q;q)_{n}a^n}{(ax;q)_{n+1}}\right)\\
    &=\frac{-(q;q)_{n}a^nD_{q}(ax;q)_{n+1}}{(ax;q)_{n+1}(aqx;q)_{n+1}}\\
    &=\frac{(q;q)_{n+1}(aqx;q)_{n}a^{n+1}}{(ax;q)_{n+1}(aqx;q)_{n+1}}\\
    &=\frac{(q;q)_{n+1}a^{n+1}}{(1-q^{n+1}xa)(ax;q)_{n+1}}\\
    &=\frac{(q;q)_{n+1}a^{n+1}}{(ax;q)_{n+2}}.
\end{align*}
The proof is reached.
\end{proof}

\begin{theorem}\label{theo_der_poch_reci}
    \begin{equation*}
        D_{q}^{m}\left\{\frac{1}{(ax;q)_{n}}\right\}=\frac{(q^n;q)_{m}a^{m}}{(ax;q)_{m+n}}.
    \end{equation*}
\end{theorem}
\begin{proof}
    \begin{align*}
        D_{q}^{m}\left\{\frac{1}{(ax;q)_{n}}\right\}&=\frac{1}{a^{n-1}(q;q)_{n-1}}D_{q}^{m+n-1}\left\{\frac{1}{1-ax}\right\}\\
        &=\frac{1}{a^{n-1}(q;q)_{n-1}}\frac{(q;q)_{m+n-1}a^{m+n-1}}{(ax;q)_{m+n}}\\
        &=\frac{(q^n;q)_{m}a^{m}}{(ax;q)_{m+n}}.
    \end{align*}
\end{proof}

\begin{theorem}
    \begin{equation*}
        D_{q}^{k}\left\{\frac{x^n}{(ax;q)_{n}}\right\}=\frac{(q;q)_{n}}{(ax;q)_{n}}\sum_{i=0}^{k}\qbinom{k}{i}_{q}\frac{(q^n;q)_{k-i}(ax;q)_{i}a^{k-i}}{(q;q)_{n-i}(aq^nx;q)_{k}}x^{n-i}.
    \end{equation*}
\end{theorem}
\begin{proof}
    \begin{align*}
        D_{q}^{k}\left\{\frac{x^n}{(ax;q)_{n}}\right\}&=\sum_{i=0}^{k}q^{i(i-k)}\qbinom{k}{i}_{q}D_{q}^{i}\left\{x^n\right\}D_{q}^{k-i}\left\{\frac{1}{(aq^ix;q)_{n}}\right\}\\
        &=\sum_{i=0}^{k}q^{i(i-k)}\qbinom{k}{i}_{q}\frac{(q;q)_{n}}{(q;q)_{n-i}}x^{n-i}\frac{(q^n;q)_{k-i}(aq^i)^{k-i}}{(aq^ix;q)_{n+k-i}}\\
        &=\frac{(q;q)_{n}}{(ax;q)_{n}}\sum_{i=0}^{k}\qbinom{k}{i}_{q}\frac{(q^n;q)_{k-i}(ax;q)_{i}a^{k-i}}{(q;q)_{n-i}(aq^nx;q)_{k}}x^{n-i}.
    \end{align*}
\end{proof}

\begin{theorem}\label{theo_der_poch/poch}
For $k\leq n$,
    \begin{equation*}
        D_{q}^{k}\left\{\frac{(ax;q)_{n}}{(bx;q)_{n}}\right\}=\frac{(ax,q;q)_{n}}{(bx;q)_{n}(bq^nx;q)_{k}}\sum_{i=0}^{k}\qbinom{k}{i}_{q}q^{\binom{i}{2}}(-a)^ib^{k-i}\frac{(bx;q)_{i}(q^{n};q)_{k-i}}{(q;q)_{n-i}(ax;q)_{i}}.
    \end{equation*}
\end{theorem}
\begin{proof}
From Eq.(\ref{eqn_iden4}),
    \begin{align*}
        D_{q}^{k}\left\{\frac{(ax;q)_{n}}{(bx;q)_{n}}\right\}
        &=\sum_{i=0}^{k}q^{i(i-k)}\qbinom{k}{i}_{q}D_{q}^{i}\left\{(ax;q)_{n}\right\}D_{q}^{k-i}\left\{\frac{1}{(bq^ix;q)_{n}}\right\}\\
        &=\sum_{i=0}^{k}q^{i(i-k)}\qbinom{k}{i}_{q}(-a)^iq^{\binom{i}{2}}\frac{(q;q)_{n}}{(q;q)_{n-i}}\frac{(ax;q)_{n}}{(ax;q)_{i}}\frac{(bx;q)_{i}(q^{n};q)_{k-i}(bq^i)^{k-i}}{(bx;q)_{n+k}}\\
        &=\frac{(ax,q;q)_{n}}{(bx;q)_{n}(bq^nx;q)_{k}}\sum_{i=0}^{k}\qbinom{k}{i}_{q}q^{\binom{i}{2}}(-a)^ib^{k-i}\frac{(bx;q)_{i}(q^{n};q)_{k-i}}{(q;q)_{n-i}(ax;q)_{i}}.
    \end{align*}
\end{proof}
If $a=b$, then for all $n\geq0$
\begin{equation*}
    \sum_{i=0}^{k}\qbinom{k}{i}_{q}q^{\binom{i}{2}}(-1)^i\frac{(q^{n};q)_{k-i}}{(q;q)_{n-i}}=0.
\end{equation*}

\begin{theorem}
    \begin{multline*}
        D_{q}^{k}\left\{\frac{(ax,bx;q)_{n}}{(cx,dx;q)_{n}}\right\}\\
        =(q;q)_{n}^2(ax,bx;q)_{n}\sum_{i=0}^{k}\qbinom{k}{i}_{q}\frac{q^{\binom{i}{2}}}{(bx;q)_{i}}\sum_{j=0}^{i}\qbinom{i}{j}_{q}q^{j(j-i)}\frac{(-a)^j}{(q;q)_{n-j}}\frac{(-b)^{i-j}}{(q;q)_{n-i+j}}\frac{(bq^nx;q)_{j}}{(ax;q)_{j}}\\
        \times\frac{1}{(cq^ix,dq^ix;q)_{n}(dq^{n+i}x;q)_{k-i}}\sum_{l=0}^{k-i}\qbinom{k-i}{l}_{q}\frac{(q^n;q)_{l}(q^{n};q)_{k-i-l}(dq^ix;q)_{l}}{(cq^{n+i}x;q)_{l}}c^ld^{k-i-l}.
    \end{multline*}
\end{theorem}
\begin{proof}
    \begin{align*}
        &D_{q}^{k}\left\{\frac{(ax,bx;q)_{n}}{(cx,dx;q)_{n}}\right\}\\
        &=\sum_{i=0}^{k}\qbinom{k}{i}_{q}q^{i(i-k)}D_{q}^{i}\left\{(ax,bx;q)_{n}\right\}D_{q}^{k-i}\left\{\frac{1}{(cq^ix,dq^ix;q)_{n}}\right\}\\
        &=(q;q)_{n}^2(ax,bx;q)_{n}\sum_{i=0}^{k}\qbinom{k}{i}_{q}\frac{q^{\binom{i}{2}}}{(bx;q)_{i}}\sum_{j=0}^{i}\qbinom{i}{j}_{q}q^{j(j-i)}\frac{(-a)^j}{(q;q)_{n-j}}\frac{(-b)^{i-j}}{(q;q)_{n-i+j}}\frac{(bq^nx;q)_{j}}{(ax;q)_{j}}\\
        &\times\frac{1}{(cq^ix,dq^ix;q)_{n}(dq^{n+i}x;q)_{k-i}}\sum_{l=0}^{k-i}\qbinom{k-i}{l}_{q}\frac{(q^n;q)_{l}(q^{n};q)_{k-i-l}(dq^ix;q)_{l}}{(cq^{n+i}x;q)_{l}}c^ld^{k-i-l}.
    \end{align*}
\end{proof}
Por hacer $n\rightarrow\infty$
\begin{multline*}
        D_{q}^{k}\left\{\frac{(ax,bx;q)_{\infty}}{(cx,dx;q)_{\infty}}\right\}\\
        =\frac{(ax,bx;q)_{\infty}}{(cx,dx;q)_{\infty}}\sum_{i=0}^{k}\qbinom{k}{i}_{q}\frac{q^{\binom{i}{2}}(cx,dx;q)_{i}}{(bx;q)_{i}}\sum_{j=0}^{i}\qbinom{i}{j}_{q}q^{j(j-i)}\frac{(-a)^j(-b)^{i-j}}{(ax;q)_{j}}\\
        \times\sum_{l=0}^{k-i}\qbinom{k-i}{l}_{q}(dq^ix;q)_{l}c^ld^{k-i-l}.
    \end{multline*}

\section{$q$-derivative of ${}_{r}\Phi_{s}$ with respect to parameters}

\begin{theorem}
For all $n\in\N$,
    \begin{multline*}
        D_{c_{1},q}^{k}\left\{{}_{r}\Phi_{s}\left(
        \begin{array}{c}
             a_{1},\ldots,a_{r} \\
             c_{1},c_{2},\ldots,c_{s}
        \end{array};
        q,u,z
        \right)\right\}\\
        =\frac{(-1)^{1+s-r}\prod_{i=1}^{r}(1-a_{i})(q;q)_{k}z}{(c_{1};q)_{k+1}(1-q)\prod_{i=2}^{s}(1-c_{i})}\\
        \hspace{5cm}{}_{r+1}\Phi_{s+1}\left(
        \begin{array}{c}
             a_{1}q,\cdots,a_{r}q,q^{k+1} \\
             c_{1}q^{k+1},c_{2},\cdots,c_{s},q^2
        \end{array};
        q,u,q^{1+s-r}uz
        \right).
    \end{multline*}
\end{theorem}
\begin{proof}
From Eq.(\ref{eqn_def_hyp}) and Theorem \ref{theo_der_poch_reci},
    \begin{align*}
        &D_{c_{1},q}^{k}\left\{{}_{r}\Phi_{s}\left(
        \begin{array}{c}
             a_{1},\ldots,a_{r} \\
             c_{1},c_{2},\ldots,c_{s}
        \end{array};
        q,u,z
        \right)\right\}\\
        &\hspace{1cm}=\sum_{n=0}^{\infty}u^{\binom{n}{2}}\frac{(a_{1};q)_{n}\cdots(a_{r};q)_{n}}{(c_{2};q)_{n}\cdots(c_{s};q)_{n}(q;q)_{n}}\left((-1)^nq^{\binom{n}{2}}\right)^{1+s-r}z^nD_{c_{1},q}^{k}\left\{\frac{1}{(c_{1};q)_{n}}\right\}\\
        &\hspace{1cm}=c_{1}^k\sum_{n=0}^{\infty}u^{\binom{n}{2}}\frac{(a_{1};q)_{n}\cdots(a_{r};q)_{n}(q^n;q)_{k}}{(c_{1};q)_{k+n}(c_{2};q)_{n}\cdots(c_{s};q)_{n}(q;q)_{n}}\left((-1)^nq^{\binom{n}{2}}\right)^{1+s-r}z^n.
    \end{align*}
Taking into account that $(q^0;q)_{k}=0$, we shift the index $n$ by one unit and apply the identity Eq.(\ref{eqn_iden2}),
    \begin{align*}
        &D_{c_{1},q}^{k}\left\{{}_{r}\Phi_{s}\left(
        \begin{array}{c}
             a_{1},\ldots,a_{r} \\
             c_{1},c_{2},\ldots,c_{s}
        \end{array};
        q,u,z
        \right)\right\}\\
        &=\frac{z}{(c_{1};q)_{k}}\sum_{n=0}^{\infty}u^{\binom{n+1}{2}}\frac{(a_{1};q)_{n+1}\cdots(a_{r};q)_{n+1}(q^{n+1};q)_{k}}{(c_{1}q^k;q)_{n+1}(c_{2};q)_{n+1}\cdots(c_{s};q)_{n+1}(q;q)_{n+1}}\\
        &\hspace{8cm}\left((-1)^{n+1}q^{\binom{n+1}{2}}\right)^{1+s-r}z^n\\
        &=\frac{(-1)^{1+s-r}\prod_{i=1}^{r}(1-a_{i})z}{(cx;q)_{k+1}(1-q)\prod_{i=2}^{s}(1-c_{i})}\sum_{n=0}^{\infty}\frac{u^{\binom{n}{2}}(a_{1}q;q)_{n}\cdots(a_{r}q;q)_{n}(q^{n+1};q)_{k}}{(c_{1}q^{k+1};q)_{n}(c_{2}q;q)_{n}\cdots(c_{s}q;q)_{n}(q^2;q)_{n}}\\
        &\hspace{8cm}\times\left((-1)^nq^{\binom{n}{2}}\right)^{1+s-r}(q^{1+s-r}uz)^n\\
        &=\frac{(-1)^{1+s-r}\prod_{i=1}^{r}(1-a_{i})(q;q)_{k}z}{(cx;q)_{k+1}(1-q)\prod_{i=2}^{s}(1-c_{i})}\\
        &\hspace{1cm}\times\sum_{n=0}^{\infty}\frac{u^{\binom{n}{2}}(a_{1}q;q)_{n}\cdots(a_{r}q;q)_{n}(q^{k+1};q)_{n}}{(c_{1}q^{k+1};q)_{n}(c_{2};q)_{n}\cdots(c_{s};q)_{n}(q^2;q)_{n}(q;q)_{n}}\left((-1)^nq^{\binom{n}{2}}\right)^{1+s-r}(q^{1+s-r}uz)^n\\
        &=\frac{(-1)^{1+s-r}\prod_{i=1}^{r}(1-a_{i})(q;q)_{k}z}{(c_{1};q)_{k+1}(1-q)\prod_{i=2}^{s}(1-c_{i})}\\
        &\hspace{5cm}{}_{r+1}\Phi_{s+1}\left(
        \begin{array}{c}
             a_{1}q,\cdots,a_{r}q,q^{k+1} \\
             c_{1}q^{k+1},c_{2},\cdots,c_{s},q^2
        \end{array};
        q,u,q^{1+s-r}uz
        \right).
    \end{align*}
\end{proof}

\begin{theorem}
For all $n\in\N$,
    \begin{multline*}
        D_{c,q}^{k}\left\{{}_{r}\Phi_{1}\left(
        \begin{array}{c}
             a_{1},\ldots,a_{r} \\
             c
        \end{array};
        q,u,z
        \right)\right\}\\
        =\frac{(-1)^{2-r}\prod_{i=1}^{r}(1-a_{i})(q;q)_{k}z}{(c;q)_{k+1}(1-q)}\\
        \hspace{5cm}{}_{r+1}\Phi_{2}\left(
        \begin{array}{c}
             a_{1}q,\cdots,a_{r}q,q^{k+1} \\
             cq^{k+1},q^2
        \end{array};
        q,u,q^{1+s-r}uz
        \right).
    \end{multline*}
\end{theorem}

\section{New identities of the deformed $q$-exponential operator}

Set $u\in\C$. In \cite{orozco} the $u$-deformed $q$-exponential operator $\T(yD_{q}|u)$ is defined by letting
\begin{equation*}
    \T(yD_{q}|u)=\sum_{n=0}^{\infty}u^{\binom{n}{2}}\frac{(yD_{q})^{n}}{(q;q)_{n}}. 
\end{equation*}
The $u$-deformed $q$-exponential operator $\T(yD_{q}|u)$ generalizes the operators given by Chen and Liu \cite{chen1,chen2}:
\begin{equation*}
    \T(bD_{q})=\sum_{n=0}^{\infty}\frac{(bD_{q})^n}{(q;q)_{n}}=\T(bD_{q}|1),
\end{equation*}
and by Saad et al. \cite{saad,saad2}:
\begin{equation*}
    \RR(bD_{q})=\sum_{n=0}^{\infty}(-1)^{n}\frac{(bD_{q})^nq^{\binom{n}{2}}}{(q;q)_{n}}=\T(-bD_{q}|q).
\end{equation*}
Some $q$-operator identities:
\begin{theorem}\label{theo_dexp_qpoch}
For all $n\geq0$
    \begin{equation}
        \T(yD_{q}|u)\left\{\frac{1}{(ax;q)_{n}}\right\}=\frac{1}{(ax;q)_{n}}{}_{1}\Phi_{1}\left(
        \begin{array}{c}
             q^{n} \\
             aq^{n}x
        \end{array};
        q,uq^{-1},-ay
        \right).
    \end{equation}
\end{theorem}
\begin{proof}
We have
    \begin{align*}
        \T(yD_{q}|u)\left\{\frac{1}{(ax;q)_{n}}\right\}&=\sum_{k=0}^{\infty}u^{\binom{k}{2}}\frac{y^k}{(q;q)_{k}}D_{q}^{k}\left\{\frac{1}{(ax;q)_{n}}\right\}\\
        &=\sum_{k=0}^{\infty}u^{\binom{k}{2}}\frac{y^k}{(q;q)_{k}}\frac{(q^n;q)_{k}a^k}{(ax;q)_{k+n}}\\
        &=\sum_{k=0}^{\infty}u^{\binom{k}{2}}\frac{(ay)^k}{(q;q)_{k}}\frac{(q^{n};q)_{k}}{(ax;q)_{n}(aq^{n}x;q)_{k}}\\
        &=\frac{1}{(ax;q)_{n}}{}_{1}\Phi_{1}\left(
        \begin{array}{c}
             q^{n} \\
             aq^{n}x
        \end{array};
        q,uq^{-1},-ay
        \right)
    \end{align*}
as claimed.    
\end{proof}

\begin{theorem}
For all $n\geq0$
    \begin{multline}
        \T(yD_{q}|u)\left\{\frac{(ax;q)_{n}}{(bx;q)_{n}}\right\}\\
        =\frac{(ax;q)_{n}}{(bx;q)_{n}}\sum_{i=0}^{n}(uq)^{\binom{i}{2}}\qbinom{n}{i}_{q}\frac{(bx;q)_{i}(-ay)^i}{(ax;q)_{i}(bq^nx;q)_{i}}{}_{1}\Phi_{1}\left(
     \begin{array}{c}
         q^n \\
        bq^{n+i}x
        \end{array};
       q,uq^{-1},-u^iby
        \right).
    \end{multline}
\end{theorem}
\begin{proof}
    \begin{align*}
        &\T(yD_{q}|u)\left\{\frac{(ax;q)_{n}}{(bx;q)_{n}}\right\}\\
        &=\sum_{k=0}^{\infty}u^{\binom{k}{2}}\frac{y^k}{(q;q)_{k}}D_{q}^{k}\left\{\frac{(ax;q)_{n}}{(bx;q)_{n}}\right\}\\
        &=\frac{(ax,q;q)_{n}}{(bx;q)_{n}}\sum_{k=0}^{\infty}u^{\binom{k}{2}}\frac{y^k}{(q;q)_{k}}\frac{1}{(bq^nx;q)_{k}}\sum_{i=0}^{k}\qbinom{k}{i}_{q}q^{\binom{i}{2}}(-a)^ib^{k-i}\frac{(bx;q)_{i}(q^{n};q)_{k-i}}{(q;q)_{n-i}(ax;q)_{i}}\\
        &=\frac{(ax;q)_{n}}{(bx;q)_{n}}\sum_{i=0}^{n}(uq)^{\binom{i}{2}}\qbinom{n}{i}_{q}\frac{(bx;q)_{i}(-ay)^i}{(ax;q)_{i}(bq^nx;q)_{i}}\sum_{k=0}^{\infty}u^{\binom{k}{2}}\frac{(q^{n};q)_{k}}{(q;q)_{k}(bq^{n+i}x;q)_{k}}(u^iyb)^{k}\\
        &=\frac{(ax;q)_{n}}{(bx;q)_{n}}\sum_{i=0}^{n}(uq)^{\binom{i}{2}}\qbinom{n}{i}_{q}\frac{(bx;q)_{i}(-ay)^i}{(ax;q)_{i}(bq^nx;q)_{i}}{}_{1}\Phi_{1}\left(
     \begin{array}{c}
         q^n \\
        bq^{n+i}x
        \end{array};
       q,uq^{-1},-u^iby
        \right).
    \end{align*}
\end{proof}
If $a=b$, then for all $n\geq0$
\begin{equation*}
    \sum_{i=0}^{n}(uq)^{\binom{i}{2}}\qbinom{n}{i}_{q}\frac{(-ay)^i}{(aq^nx;q)_{i}}{}_{1}\Phi_{1}\left(
     \begin{array}{c}
         q^n \\
        aq^{n+i}x
        \end{array};
       q,uq^{-1},-u^iay
        \right)=1
\end{equation*}

\begin{theorem}
    \begin{multline*}
        \T(yD_{q}|u)\left\{\frac{x^n}{(ax;q)_{n}}\right\}\\
        =\frac{1}{(ax;q)_{n}}\sum_{i=0}^{n}\qbinom{n}{i}_{q}u^{\binom{i}{2}}\frac{(ax;q)_{i}}{(aq^{n}x;q)_{i}}x^{n-i}y^{i}{}_{1}\Phi_{1}\left(
        \begin{array}{c}
             q^{n} \\
             aq^{n+i}x
        \end{array};
        q,uq^{-1},-u^iay
        \right).
    \end{multline*}
\end{theorem}
\begin{proof}
    \begin{align*}
        &\T(yD_{q}|u)\left\{\frac{x^n}{(ax;q)_{n}}\right\}\\
        &=\sum_{k=0}^{\infty}u^{\binom{k}{2}}\frac{y^k}{(q;q)_{k}}D_{q}^{k}\left\{\frac{x^n}{(ax;q)_{n}}\right\}\\
        &=\frac{(q;q)_{n}}{(ax;q)_{n}}\sum_{k=0}^{\infty}u^{\binom{k}{2}}\frac{y^k}{(q;q)_{k}}\sum_{i=0}^{k}\qbinom{k}{i}_{q}\frac{(q^n;q)_{k-i}(ax;q)_{i}a^{k-i}}{(q;q)_{n-i}(aq^nx;q)_{k}}x^{n-i}\\
        &=\frac{1}{(ax;q)_{n}}\sum_{i=0}^{n}\qbinom{n}{i}_{q}u^{\binom{i}{2}}\frac{(ax;q)_{i}}{(aq^{n}x;q)_{i}}x^{n-i}y^{i}\sum_{k=0}^{\infty}u^{\binom{k}{2}}\frac{(q^n;q)_{k}(u^iay)^{k}}{(q;q)_{k}(aq^{n+i}x;q)_{k}}\\
        &=\frac{1}{(ax;q)_{n}}\sum_{i=0}^{n}\qbinom{n}{i}_{q}u^{\binom{i}{2}}\frac{(ax;q)_{i}}{(aq^{n}x;q)_{i}}x^{n-i}y^{i}{}_{1}\Phi_{1}\left(
        \begin{array}{c}
             q^{n} \\
             aq^{n+i}x
        \end{array};
        q,uq^{-1},-u^iay
        \right).
    \end{align*}
\end{proof}

\begin{theorem}\label{theo_dexp_1Phi1}
For all $n\geq0$
   \begin{multline}
      \T(zD_{q}|v)\left\{\frac{1}{(ax;q)_{n}}{}_{1}\Phi_{1}\left(
        \begin{array}{c}
             q^{n} \\
             aq^{n}x
        \end{array};
        q,uq^{-1},-ay
        \right)\right\}\\
        =\frac{1}{(ax;q)_{n}}\sum_{k=0}^{\infty}u^{\binom{k}{2}}\frac{(q^n;q)_{k}(ay)^k}{(q;q)_{k}(axq^n;q)_{k}}{}_{1}\Phi_{1}\left(
        \begin{array}{c}
             q^{n+k}\\
             aq^{n+k}x
        \end{array};
        q,vq^{-1},-az
        \right).
    \end{multline}
\end{theorem}
\begin{proof}
From Theorem \ref{theo_dexp_qpoch}, we have
    \begin{align*}
        &\T(zD_{q}|v)\left\{\frac{1}{(ax;q)_{n}}{}_{1}\Phi_{1}\left(
        \begin{array}{c}
             q^{n} \\
             aq^{n}x
        \end{array};
        q,uq^{-1},-ay
        \right)\right\}\\
        &\hspace{3cm}=\T(zD_{q}|v)\left\{\sum_{k=0}^{\infty}u^{\binom{k}{2}}\frac{(q^n;q)_{k}}{(q;q)_{k}(ax;q)_{n+k}}(-ay)^k\right\}\\
        &\hspace{3cm}=\sum_{k=0}^{\infty}u^{\binom{k}{2}}\frac{(q^n;q)_{k}}{(q;q)_{k}}(-ay)^k\T(zD_{q}|v)\left\{\frac{1}{(ax;q)_{n+k}}\right\}\\
        &\hspace{3cm}=\frac{1}{(ax;q)_{n}}\sum_{k=0}^{\infty}u^{\binom{k}{2}}\frac{(q^n;q)_{k}(ay)^k}{(q;q)_{k}(axq^n;q)_{k}}{}_{1}\Phi_{1}\left(
        \begin{array}{c}
             q^{n+k}\\
             aq^{n+k}x
        \end{array};
        q,vq^{-1},-az
        \right)
    \end{align*}
as claimed.
\end{proof}

\begin{theorem}
   \begin{multline*}
      \T(yD_{q}|v)\left\{{}_{r}\Phi_{s}\left(
     \begin{array}{c}
         a_{1},\ldots,a_{r} \\
        b_{1}x,b_{2},\ldots,b_{s}
        \end{array};
       q,u,z
        \right)\right\}\\=\sum_{n=0}^{\infty}u^{\binom{n}{2}}\frac{(a_{1},a_{2},\ldots,a_{r};q)_{n}}{(b_{1}x,b_{2},\ldots,b_{s};q)_{n}(q;q)_{n}}\\
        {}_{1}\Phi_{1}\left(
        \begin{array}{c}
             q^{n} \\
             q^{n}b_{1}x
        \end{array};
        q,vq^{-1},-b_{1}y
        \right)\bigg((-1)^nq^{\binom{n}{2}}\bigg)^{1+s-r}z^n.
    \end{multline*}
\end{theorem}
\begin{proof}
    \begin{align*}
        &\T(yD_{q}|v)\left\{{}_{r}\Phi_{s}\left(
     \begin{array}{c}
         a_{1},a_{2},\ldots,a_{r} \\
        b_{1}x,b_{2},\ldots,b_{s}
        \end{array};
       q,u,z
        \right)\right\}\\
        &=\T(yD_{q}|v)\left\{\sum_{n=0}^{\infty}u^{\binom{n}{2}}\frac{(a_{1},a_{2},\ldots,a_{r};q)_{n}}{(b_{1}x,b_{2},\ldots,b_{s};q)_{n}(q;q)_{n}}\bigg((-1)^nq^{\binom{n}{2}}\bigg)^{1+s-r}z^n\right\}\\
        &=\sum_{n=0}^{\infty}u^{\binom{n}{2}}\T(yD_{q}|v)\left\{\frac{1}{(b_{1}x;q)_{n}}\right\}\frac{(a_{1},a_{2},\ldots,a_{r};q)_{n}}{(b_{2},\ldots,b_{s};q)_{n}(q;q)_{n}}\bigg((-1)^nq^{\binom{n}{2}}\bigg)^{1+s-r}z^n\\
        &=\sum_{n=0}^{\infty}u^{\binom{n}{2}}\frac{(a_{1},a_{2},\ldots,a_{r};q)_{n}}{(b_{1}x,b_{2},\ldots,b_{s};q)_{n}(q;q)_{n}}\\
        &\hspace{3cm}{}_{1}\Phi_{1}\left(
        \begin{array}{c}
             q^{n}\\
             q^{n}b_{1}x
        \end{array};
        q,vq^{-1},-b_{1}y
        \right)\bigg((-1)^nq^{\binom{n}{2}}\bigg)^{1+s-r}z^n.
    \end{align*}
\end{proof}

\section{New $q$-identities from old one}

\subsection{$q$-identity from the $q$-Gauss sum}

The $q$-Gauss sum is given by [Gasper]:
\begin{equation}\label{eqn_qgauss}
    {}_{2}\phi_{1}\left(\begin{array}{c}
         a,b\\
         c 
    \end{array};q,c/ab\right)=\frac{(c/a,c/b;q)_{\infty}}{(c,c/ab;q)_{\infty}}.
\end{equation}

\begin{theorem}
    \begin{multline*}
        \frac{(q;q)_{k}}{(1-q)(c;q)_{k}}\sum_{i=0}^{k}\qbinom{k}{i}_{q}\frac{(1-aq^i)(1-bq^i)(c;q)_{i}(a;q)_{i}(b;q)_{i}}{(1-cq^{k+1})(cq^k;q)_{i}(q;q)_{i}}\left(\frac{1}{ab}\right)^n\\
        \hspace{3cm}\times{}_{4}\phi_{3}\left(\begin{array}{c}
         aq^{i+1},bq^{i+1},q^{k+1},q\\
         cq^{k+i+1},q^{i+1},q^2 
        \end{array};q,c/ab\right)\\
        =\frac{(ax,bx;q)_{\infty}}{(cx,dx;q)_{\infty}}\sum_{i=0}^{k}\qbinom{k}{i}_{q}\frac{q^{\binom{i}{2}}(cx,dx;q)_{i}}{(bx;q)_{i}}\sum_{j=0}^{i}\qbinom{i}{j}_{q}q^{j(j-i)}\frac{(-a)^j(-b)^{i-j}}{(ax;q)_{j}}\\
        \times\sum_{l=0}^{k-i}\qbinom{k-i}{l}_{q}(dq^ix;q)_{l}c^ld^{k-i-l}.
    \end{multline*}
\end{theorem}
\begin{proof}
On the one hand,    
    \begin{align*}
        &D_{c,q}^{k}\left\{{}_{2}\phi_{1}\left(\begin{array}{c}
         a,b\\
         c 
        \end{array};q,c/ab\right)\right\}\\
        &\hspace{1cm}=\sum_{n=0}^{\infty}\frac{(a;q)_{n}(b;q)_{n}}{(q;q)_{n}}\left(\frac{1}{ab}\right)^nD_{q}^{k}\left\{\frac{c^n}{(c;q)_{n}}\right\}\\
        &\hspace{1cm}=\sum_{n=0}^{\infty}\frac{(a;q)_{n}(b;q)_{n}}{(q;q)_{n}}\left(\frac{1}{ab}\right)^n\frac{(q;q)_{n}}{(c;q)_{k}(cq^k;q)_{n}}\sum_{i=0}^{k}\qbinom{k}{i}_{q}\frac{(q^n;q)_{k-i}(c;q)_{i}}{(q;q)_{n-i}}c^{n-i}\\
        &\hspace{1cm}=\frac{1}{(c;q)_{k}}\sum_{i=0}^{k}\qbinom{k}{i}_{q}(c;q)_{i}c^{-i}\sum_{n=i}^{\infty}\frac{(a;q)_{n}(b;q)_{n}}{(cq^k;q)_{n}}\frac{(q^n;q)_{k-i}}{(q;q)_{n-i}}\left(\frac{c}{ab}\right)^n\\
        &\hspace{1cm}=\frac{1}{(c;q)_{k}}\sum_{i=0}^{k}\qbinom{k}{i}_{q}(c;q)_{i}c^{-i}\sum_{n=0}^{\infty}\frac{(a;q)_{n+i}(b;q)_{n+i}}{(cq^k;q)_{n+i}}\frac{(q^{n+i};q)_{k-i}}{(q;q)_{n}}\left(\frac{c}{ab}\right)^{n+i}\\
        &=\frac{1}{(c;q)_{k}}\sum_{i=0}^{k}\qbinom{k}{i}_{q}\frac{(c;q)_{i}(a;q)_{i}(b;q)_{i}}{(cq^k;q)_{i}}\left(\frac{1}{ab}\right)^n\sum_{n=0}^{\infty}\frac{(aq^i;q)_{n}(bq^i;q)_{n}}{(cq^{k+i};q)_{n}}\frac{(q^{n+i};q)_{k-i}}{(q;q)_{n}}\left(\frac{c}{ab}\right)^{n}.
    \end{align*}
As $(q^{n+i};q)_{k-i}=\frac{(q^{n};q)_{k}}{(q^n;q)_{i}}$ and $(q^0;q)_{m}=0$, then using Eq.(\ref{eqn_iden3}) we have
    \begin{align*}
        &D_{c,q}^{k}\left\{{}_{2}\phi_{1}\left(\begin{array}{c}
         a,b\\
         c 
        \end{array};q,c/ab\right)\right\}\\
        &=\frac{1}{(c;q)_{k}}\sum_{i=0}^{k}\qbinom{k}{i}_{q}\frac{(c;q)_{i}(a;q)_{i}(b;q)_{i}}{(cq^k;q)_{i}}\left(\frac{1}{ab}\right)^n\sum_{n=0}^{\infty}\frac{(aq^i;q)_{n+1}(bq^i;q)_{n+1}}{(cq^{k+i};q)_{n+1}(q;q)_{n+1}}\frac{(q^{n+1};q)_{k}}{(q^{n+1};q)_{i}}\left(\frac{c}{ab}\right)^{n}\\
        &=\frac{(q;q)_{k}}{(1-q)(c;q)_{k}}\sum_{i=0}^{k}\qbinom{k}{i}_{q}\frac{(1-aq^i)(1-bq^i)(c;q)_{i}(a;q)_{i}(b;q)_{i}}{(1-cq^{k+1})(cq^k;q)_{i}(q;q)_{i}}\left(\frac{1}{ab}\right)^n\\
        &\hspace{7cm}\times\sum_{n=0}^{\infty}\frac{(aq^{i+1};q)_{n}(bq^{i+1};q)_{n}}{(cq^{k+i+1};q)_{n}(q^2;q)_{n}}\frac{(q^{k+1};q)_{n}}{(q^{i+1};q)_{n}}\left(\frac{c}{ab}\right)^{n}\\
        &=\frac{(q;q)_{k}}{(1-q)(c;q)_{k}}\sum_{i=0}^{k}\qbinom{k}{i}_{q}\frac{(1-aq^i)(1-bq^i)(c;q)_{i}(a;q)_{i}(b;q)_{i}}{(1-cq^{k+1})(cq^k;q)_{i}(q;q)_{i}}\left(\frac{1}{ab}\right)^n\\
        &\hspace{7cm}\times{}_{4}\phi_{3}\left(\begin{array}{c}
         aq^{i+1},bq^{i+1},q^{k+1},q\\
         cq^{k+i+1},q^{i+1},q^2 
        \end{array};q,c/ab\right).
    \end{align*}
On the other hand,
    \begin{multline*}
        D_{q}^{k}\left\{\frac{(ax,bx;q)_{\infty}}{(cx,dx;q)_{\infty}}\right\}\\
        =\frac{(ax,bx;q)_{\infty}}{(cx,dx;q)_{\infty}}\sum_{i=0}^{k}\qbinom{k}{i}_{q}\frac{q^{\binom{i}{2}}(cx,dx;q)_{i}}{(bx;q)_{i}}\sum_{j=0}^{i}\qbinom{i}{j}_{q}q^{j(j-i)}\frac{(-a)^j(-b)^{i-j}}{(ax;q)_{j}}\\
        \times\sum_{l=0}^{k-i}\qbinom{k-i}{l}_{q}(dq^ix;q)_{l}c^ld^{k-i-l}.
    \end{multline*}
\end{proof}

\subsection{$q$-identities from $q$-Chu-Vandermonde's sum}
The $q$-Chu-Vandermonde's sum \cite{gasper} is given by:
\begin{equation}\label{eqn_qchu}
    {}_{2}\phi_{1}(q^{-n},a;c;q,q)=\frac{(c/a;q)_{n}}{(c;q)_{n}}a^n.
\end{equation}

\begin{theorem}
For $k\leq n$,
\begin{multline*}
    {}_{3}\phi_{2}\left(
        \begin{array}{c}
             q^{-n+1},aq,q^{k+1}\\
             cq^{k+1},q^{2}
        \end{array};
        q,q
        \right)\\
        =a^n\frac{(1-q)(q;q)_{n}(c;q)_{k+1}(c/a;q)_{n}}{q(q;q)_{k}(c;q)_{n}(q^nc;q)_{k}}\sum_{i=0}^{k}\qbinom{k}{i}_{q}q^{\binom{i}{2}}(-1/a)^i\frac{(c;q)_{i}(q^{n};q)_{k-i}}{(q;q)_{n-i}(c/a;q)_{i}}.
\end{multline*}
\end{theorem}
\begin{proof}
On the one hand,
    \begin{align*}
        D_{c,q}^{k}\left\{{}_{2}\phi_{1}\left(
        \begin{array}{c}
             q^{-n},a \\
             c
        \end{array};
        q,q
        \right)\right\}
        =\frac{(q;q)_{k}q}{(c;q)_{k+1}(1-q)}
        {}_{3}\phi_{2}\left(
        \begin{array}{c}
             q^{-n+1},aq,q^{k+1} \\
             cq^{k+1},q^2
        \end{array};
        q,q
        \right).
    \end{align*}
On the other hand,
    \begin{equation*}
        D_{c,q}^{k}\left\{\frac{(c/a;q)_{n}}{(c;q)_{n}}\right\}=\frac{(c/a,q;q)_{n}}{(c;q)_{n}(q^nc;q)_{k}}\sum_{i=0}^{k}\qbinom{k}{i}_{q}q^{\binom{i}{2}}(-1/a)^i\frac{(c;q)_{i}(q^{n};q)_{k-i}}{(q;q)_{n-i}(c/a;q)_{i}}.
    \end{equation*}
    
\end{proof}
Set $a=1$ in the above theorem. Then for all $c\in\R$,
\begin{equation*}
    {}_{3}\phi_{2}\left(
        \begin{array}{c}
             q^{-n+1},q,q^{k+1} \\
             cq^{k+1},q^2
        \end{array};
        q,q
        \right)=0
\end{equation*}

\begin{theorem}
    \begin{multline*}
        \sum_{k=0}^{n}\frac{(q^{-n},a;q)_{k}}{(c;q)_{k}(q;q)_{k}}
        {}_{1}\Phi_{1}\left(
        \begin{array}{c}
             q^{k} \\
             q^{k}c
        \end{array};
        q,uq^{-1},-y
        \right)q^k\\
        =\frac{(c/a;q)_{n}}{(c;q)_{n}}\sum_{i=0}^{n}(uq)^{\binom{i}{2}}\qbinom{n}{i}_{q}\frac{(c;q)_{i}(-a^{-1}y)^i}{(c/a;q)_{i}(q^nc;q)_{i}}{}_{1}\Phi_{1}\left(
     \begin{array}{c}
         q^n \\
        q^{n+i}c
        \end{array};
       q,uq^{-1},-u^iy
        \right).
    \end{multline*}
\end{theorem}
\begin{proof}
On the one hand,
    \begin{align*}
      &\T(yD_{c,q}|u)\left\{{}_{2}\phi_{1}\left(
     \begin{array}{c}
         q^{-n},a \\
        c
        \end{array};
       q,q
        \right)\right\}\\
        &\hspace{1cm}=\sum_{k=0}^{n}\frac{(q^{-n},a;q)_{k}}{(q;q)_{k}}\T(yD_{c,q}|u)\left\{\frac{1}{(c;q)_{k}}\right\}
        q^k\\
        &\hspace{1cm}=\sum_{k=0}^{n}\frac{(q^{-n},a;q)_{k}}{(c;q)_{k}(q;q)_{k}}
        {}_{1}\Phi_{1}\left(
        \begin{array}{c}
             q^{k} \\
             q^{k}c
        \end{array};
        q,uq^{-1},-y
        \right)q^k
    \end{align*}
On the other hand,
    \begin{multline}
        \T(yD_{c,q}|u)\left\{\frac{(c/a;q)_{n}}{(c;q)_{n}}\right\}\\
        =\frac{(c/a;q)_{n}}{(c;q)_{n}}\sum_{i=0}^{n}(uq)^{\binom{i}{2}}\qbinom{n}{i}_{q}\frac{(c;q)_{i}(-a^{-1}y)^i}{(c/a;q)_{i}(q^nc;q)_{i}}{}_{1}\Phi_{1}\left(
     \begin{array}{c}
         q^n \\
        q^{n+i}c
        \end{array};
       q,uq^{-1},-u^iy
        \right).
    \end{multline}    
\end{proof}

\subsection{$q$-identity from Jackson's transformation formula}

Jackson's transformation of ${}_{2}\phi_{1}$-series is
\begin{equation}\label{eqn_jackson}
    {}_{2}\phi_{1}\left(\begin{array}{c}
         a,b\\
         c 
    \end{array};q,z\right)=\frac{(az;q)_{\infty}}{(z;q)_{\infty}}{}_{2}\phi_{2}\left(\begin{array}{c}
         a,c/b\\
         c,az
    \end{array};q,bz\right).
\end{equation}

\begin{theorem}
    \begin{multline*}
        {}_{3}\phi_{2}\left(
        \begin{array}{c}
             aq,bq,q^{k+1}\\ 
             cq^{k+1},q^2
        \end{array};
        q,z
        \right)\\
        =-\frac{b(1-cq^k)(az;q)_{\infty}}{(1-a)(1-b)(z;q)_{\infty}}\sum_{i=0}^{k}\qbinom{k}{i}_{q}q^{2\binom{i}{2}}\frac{(1-aq^i)(1-cq^i/b)(1-q^{i+1})(a;q)_{i}(c;q)_{i}}{(1-azq^i)(1-cq^{k+i})(az;q)_{i}(cq^k;q)_{i}}(qz)^i\\
        \hspace{1cm}\times{}_{4}\phi_{4}\left(
        \begin{array}{c}
             aq^{i+1},cq^{i+1}/b,q^{i+2},q^{k+1} \\
             azq^{i+1},cq^{k+i+1},q^2,q^{i+1}
        \end{array};
        q,q^{i+1}bz
        \right).
    \end{multline*}
\end{theorem}
\begin{proof}
On the one hand,
    \begin{align*}
        D_{c,q}^{k}\left\{{}_{2}\phi_{1}\left(
        \begin{array}{c}
             a,b \\
             c
        \end{array};
        q,z
        \right)\right\}
        =\frac{(1-a)(1-b)(q;q)_{k}z}{(c;q)_{k+1}(1-q)}
        {}_{3}\phi_{2}\left(
        \begin{array}{c}
             aq,bq,q^{k+1}\\ 
             cq^{k+1},q^2
        \end{array};
        q,z
        \right).    
    \end{align*}
On the other hand,
    \begin{align*}
        &D_{c,q}^{k}\left\{{}_{2}\phi_{2}\left(
        \begin{array}{c}
             a,c/b \\
             c,az
        \end{array};
        q,bz
        \right)\right\}\\
        &=\sum_{n=0}^{\infty}(-1)^nq^{\binom{n}{2}}\frac{(a;q)_{n}}{(az;q)_{n}}(bz)^nD_{c,q}^{k}\left\{\frac{(c/b;q)_{n}}{(c;q)_{n}}\right\}\\
        &=\sum_{n=0}^{\infty}(-1)^nq^{\binom{n}{2}}\frac{(a;q)_{n}}{(az;q)_{n}}(bz)^n\frac{(c/b;q)_{n}(q;q)_{n}}{(c;q)_{n}(cq^n;q)_{k}}\sum_{i=0}^{k}\qbinom{k}{i}_{q}q^{\binom{i}{2}}(-1/b)^i\frac{(c;q)_{i}(q^{n};q)_{k-i}}{(q;q)_{n-i}(c/b;q)_{i}}\\
        &=\sum_{i=0}^{k}\qbinom{k}{i}_{q}q^{\binom{i}{2}}(-1/b)^i\frac{(c;q)_{i}}{(c/b;q)_{i}}\sum_{n=i}^{\infty}(-1)^nq^{\binom{n}{2}}\frac{(a;q)_{n}(c/b;q)_{n}(q;q)_{n}(q^{n};q)_{k-i}}{(az;q)_{n}(c;q)_{n}(cq^n;q)_{k}(q;q)_{n-i}}(bz)^n.
    \end{align*}
From Identities Eq.(\ref{eqn_iden2}) and Eq.(\ref{eqn_iden3}) and simplifying it, we have
    \begin{align*}
        &D_{c,q}^{k}\left\{{}_{2}\phi_{2}\left(
        \begin{array}{c}
             a,c/b \\
             c,az
        \end{array};
        q,bz
        \right)\right\}\\
        &=\frac{1}{(c;q)_{k}}\sum_{i=0}^{k}\qbinom{k}{i}_{q}q^{2\binom{i}{2}}\frac{(a;q)_{i}(c;q)_{i}(q;q)_{i}}{(az;q)_{i}(cq^k;q)_{i}}z^i\\
        &\hspace{3cm}\times\sum_{n=0}^{\infty}(-1)^{n}q^{\binom{n}{2}}\frac{(aq^{i};q)_{n}(cq^i/b;q)_{n}(q^{i+1};q)_{n}(q^{n};q)_{k}}{(azq^i;q)_{n}(q;q)_{n}(cq^{k+i};q)_{n}(q^n;q)_{i}}(q^ibz)^{n}.
    \end{align*}
As $(q^0;q)_{n}=0$, then using Eq.(\ref{eqn_iden2}) we have
    \begin{align*}
        &D_{c,q}^{k}\left\{{}_{2}\phi_{2}\left(
        \begin{array}{c}
             a,c/b \\
             c,az
        \end{array};
        q,bz
        \right)\right\}\\
        &=\frac{1}{(c;q)_{k}}\sum_{i=0}^{k}\qbinom{k}{i}_{q}q^{2\binom{i}{2}}\frac{(a;q)_{i}(c;q)_{i}(q;q)_{i}}{(az;q)_{i}(cq^k;q)_{i}}z^i\\
        &\hspace{1cm}\times\sum_{n=0}^{\infty}(-1)^{n+1}q^{\binom{n+1}{2}}\frac{(aq^{i};q)_{n+1}(cq^i/b;q)_{n+1}(q^{i+1};q)_{n+1}(q^{n+1};q)_{k}}{(azq^i;q)_{n+1}(q;q)_{n+1}(cq^{k+i};q)_{n+1}(q^{n+1};q)_{i}}(q^ibz)^{n+1}\\
        &=-\frac{bz(q;q)_{k}}{(1-q)(c;q)_{k}}\sum_{i=0}^{k}\qbinom{k}{i}_{q}q^{2\binom{i}{2}}\frac{(1-aq^i)(1-cq^i/b)(1-q^{i+1})(a;q)_{i}(c;q)_{i}}{(1-azq^i)(1-cq^{k+i})(az;q)_{i}(cq^k;q)_{i}}(qz)^i\\
        &\hspace{1cm}\times\sum_{n=0}^{\infty}(-1)^{n}q^{\binom{n}{2}}\frac{(aq^{i+1};q)_{n}(cq^{i+1}/b;q)_{n}(q^{i+2};q)_{n}(q^{k+1};q)_{n}}{(azq^{i+1};q)_{n}(q^2;q)_{n}(cq^{k+i+1};q)_{n}(q^{i+1};q)_{n}}(q^{i+1}bz)^{n}\\
        &=-\frac{bz(q;q)_{k}}{(1-q)(c;q)_{k}}\sum_{i=0}^{k}\qbinom{k}{i}_{q}q^{2\binom{i}{2}}\frac{(1-aq^i)(1-cq^i/b)(1-q^{i+1})(a;q)_{i}(c;q)_{i}}{(1-azq^i)(1-cq^{k+i})(az;q)_{i}(cq^k;q)_{i}}(qz)^i\\
        &\hspace{1cm}\times{}_{4}\phi_{4}\left(
        \begin{array}{c}
             aq^{i+1},cq^{i+1}/b,q^{i+2},q^{k+1} \\
             azq^{i+1},cq^{k+i+1},q^2,q^{i+1}
        \end{array};
        q,q^{i+1}bz
        \right).
    \end{align*}
\end{proof}

\end{document}